\documentclass[runningheads,a4paper]{llncs}

\usepackage[dvips]{epsfig}
\usepackage{amsfonts}
\usepackage{amssymb}
\usepackage{amscd}
\usepackage{amstext}
\usepackage{amsmath}
\usepackage{pstricks}
\usepackage{enumerate}

\newcommand{\calA}{{\mathcal A}}

\newcommand{\calH}{{\mathcal H}}

\newcommand{\calD}{{\mathcal D}}
\newcommand{\calP}{{\mathcal P}}

\newcommand{\calS}{{\mathcal S}}
\newcommand{\calU}{{\mathcal U}}

\newcommand{\N}{{\mathbb N}}

\newcommand{\Q}{{\mathbb Q}}

\newcommand{\Frac}[2]{\displaystyle \frac{#1}{#2}}
\newcommand{\Sum}[2]{\displaystyle{\sum_{#1}^{#2}}}
\newcommand{\Prod}[2]{\displaystyle{\prod_{#1}^{#2}}}

\def\pol#1{\langle #1 \rangle}
\def\Lyn{{\mathcal Lyn}}

\def\shuffle{\mathop{_{^{\sqcup\!\sqcup}}}}

\gdef\stuffle{\;%
  \setlength{\unitlength}{0.0125cm}%
  \begin{picture}(20,10)(220,580) 
  \thinlines 
  \put(220,592){\line( 0,-1){ 10}} 
  \put(220,582){\line( 1, 0){ 20}} 
  \put(240,582){\line( 0, 1){ 10}} 
  \put(230,592){\line( 0,-1){ 10}} 
  \put(225,587){\line( 1, 0){ 10}} 
  \end{picture}\; 
}

\def\deg{\mathrm{deg}}

\newcommand{\poly}[2]{#1 \langle #2 \rangle}

\def\QY{\poly{\Q}{Y}}

\def\kY{\poly{{\bf k}}{Y}}

\newcommand{\serie}[2]{#1 \langle \! \langle #2 \rangle \! \rangle}

\def\kYY{\serie{{\bf k}}{Y}}

\def\pol#1{\langle #1 \rangle}

\def\Lie{{\cal L}ie}

\def\LQX{\Lie_{\Q} \langle X \rangle}
\def\LQY{\Lie_{\Q} \langle Y \rangle}

\def\up#1{\raise 1ex\hbox{\footnotesize#1}}

\def\Sum{\displaystyle\sum}
\def\Prod{\displaystyle\prod}
\def\Frac{\displaystyle\frac}
\def\path{\rightsquigarrow}

\def\bv{\mid}

\gdef\ministuffle{{\scriptstyle \stuffle}}

\def\scal#1#2{\langle #1\bv#2 \rangle}

\def\deg{\mathop\mathrm{deg}\nolimits}

\def\supp{\mathop\mathrm{supp}\nolimits}

\begingroup
\def\2#1{\ifnum#1<10 0\fi\the#1}
\xdef\isodayandtime{
{\2\day-\2\month-\the\year\space\2{\count0}:%
\2{\count2}}}
\endgroup

\title{Sch\"utzenberger's factorization on the\\ 
(completed) Hopf algebra of $q-$stuffle product}
\titlerunning{Hopf algebra of $q-$stuffle product}
\author{C. Bui$^{\flat}$, G. H. E. Duchamp$^{\sharp}$, V. Hoang Ngoc Minh$^{\lozenge,\sharp}$}

\date{}

\authorrunning{V.C. Bui, G. H. E. Duchamp, V. Hoang Ngoc Minh}
\institute{$^{\flat}$Hu\'e University - College of sciences, 77 - Nguyen Hue street - Hu\'e city, Vi\^et Nam\\
$^{\sharp}$Institut Galil\'ee, LIPN - UMR 7030, CNRS - Universit\'e Paris 13, F-93430 Villetaneuse, France,\\
$^{\lozenge}$Universit\'e Lille II, 1, Place D\'eliot, 59024 Lille, France}

\begin{document}

\maketitle

\begin{abstract}
In order to extend the Sch\"utzenberger's factorization, the combinatorial Hopf algebra of the $q$-stuffles product is developed systematically in a parallel way with that of the shuffle product and and in emphasizing the Lie elements as studied
by Ree. In particular, we will give here an effective construction of pair of bases in duality. 
[\isodayandtime]

\end{abstract}

\section{Introduction}

Sch\"utzenberger's factorization \cite{schutz,reutenauer} has been introduced and plays a central role in the renormalization \cite{acta} of associators\footnote{The associators were introduced in quantum field theory by Drinfel'd \cite{drinfeld1,drinfeld2} and the universal  Drinfel'd associator, {\it i.e.} $\Phi_{KZ}$, was obtained,  in \cite{lemurakami}, with explicit coefficients which are polyz\^etas and regularized polyz\^etas (see \cite{acta} for the computation of the other associators involving only convergent polyz\^etas as local coordinates, and for three algorithmical process to regularize the divergent polyz\^etas).} which are formal power series in  non commutative variables \cite{berstel_reutenauer}. The coefficients of these power series are polynomial at positive integral multi-indices of Riemann's z\^eta function\footnote{These values are usually abbreviated MZV's by Zagier \cite{zagier} and are also called polyz\^etas by Cartier \cite{cartier2}.}  \cite{lemurakami,zagier} and they satisfy quadratic relations \cite{cartier2} which can be explained through the Lyndon words  \cite{berstel_perrin,lothaire,lyndon,radford}.

These quadratic relations can be obtained by identification of the local coordinates, in infinite dimension, on a bridge equation connecting the Cauchy and Hadamard algebras of the polylogarithmic functions and using the factorizations, by Lyndon words,  of the non commutative generating series of polylogarithms \cite{FPSAC99} and of harmonic sums \cite{acta}. This bridge equation is mainly a consequence of the double isomorphy between  these algebraic structures to respectively the shuffle \cite{FPSAC99} and quasi-shuffle (or stuffle) \cite{words03} algebras both admitting the Lyndon words as a transcendence basis\footnote{Our method applies also to any other transcendence basis built by duality from PBW, see below.} \cite{radford,MalvenutoReutenauer}.

In order to better understand the mechanisms of the shuffle product and to obtain algorithms on quasi-shuffle products, we will examine, in the section below, the commutative $q$-stuffle product interpolating between the shuffle \cite{ree}, quasi-shuffle (or stuffle \cite{MalvenutoReutenauer}) and minus-stuffle products \cite{costermans,JSC}, obtained for\footnote{In \cite{costermans}, the letter $\lambda$ is used instead of $q$.} $q=0,1$ and $-1$ respectively. We will extend the Sch\"utzenberger's factorization by developping the combinatorial Hopf algebra of this product in a parallel way with that of the shuffle and in emphasizing  the Lie elements studied by Ree \cite{ree}. In particular, we will give an effective construction (implemented in Maple \cite{Bui}) of pair of bases in duality (see Propositions \ref{checkSigma} and \ref{increasinglyndon}).

This construction uses essentially an adapted version of the Eulerian projector and its adjoint \cite{reutenauer} in order to obtain the primitive elements of the $q$-stuffle Hopf algebra (see Definition \ref{pi1}). They are obtained thanks to the computation of the logarithm of the diagonal series (see Proposition \ref{logD}). This study completes the treatement for the stuffle \cite{acta} and boils down to the shuffle case for $q=0$ \cite{reutenauer}.

Let us remark that it is quite different from other studies \cite{NSFC3,hoffman} concerning non commutative $q$-shuffle products interpolating between the concatenation and shuffle products, for $q=0$ and $1$ respectively and using the $q$-deformation theory of non commutative symmetric functions\footnote{Recall also that the algebra of non commutative symmetric functions, denoted by $\bf Sym$ is the Solomon descent algebra \cite{solomon} and it is dual to  the algebra of quasi-symmetric functions, denoted by $\bf QSym$ which is isomorphic to the quasi-shuffle algebra \cite{MalvenutoReutenauer}.\\
Thus our construction of pair of bases in duality are also suitable for $\bf Sym$ and $\bf QSym$ (and their deformations, provided they remain graded connected cocommutative Hopf algebras).} \cite{NSFC3}.

\section{$q$-deformed stuffle}

\subsection{Results for the $q$-deformed stuffle}
Let ${\bf k}$ be a unitary $\Q$-algebra containing $q$.
Let also $Y=\{y_s\}_{ s\ge1}$ be an alphabet with the total order
\begin{eqnarray}
y_1>y_2>\cdots.
\end{eqnarray}

One defines the $q$-stuffle, by a recursion or by its dual co-product $\Delta_{\stuffle_q}$, as follows. For any $y_s,y_t\in Y$ and for any $u,v\in Y^*$,
 \begin{eqnarray}\label{Delta}
   u\stuffle_q 1_{Y^*}=1_{Y^*}\stuffle_q u=u&\mbox{and}&
   y_su\stuffle_q y_tv=y_s(u\stuffle_q y_tv)+y_t(y_su\stuffle_q v) +qy_{s+t}(u\stuffle_q v),\\
   \Delta_{\stuffle_q}(1_{Y^*})=1_{Y^*}\otimes1_{Y^*}&\mbox{and}&
   \Delta_{\stuffle_q}(y_s)=y_s\otimes 1_{Y^*} +1_{Y^*}\otimes y_s +q\sum_{s_1+s_2=s}{y_{s_1}\otimes y_{s_2}}.
 \end{eqnarray}
This product is commutative, associative and unital  (the neutral being the empty word $1_{Y^*}$).
With the co-unit defined by, for any $P\in\kY$, 
\begin{eqnarray}
\epsilon(P)=\scal{P}{1_{Y^*}}
\end{eqnarray}
one gets $\calH_{{\stuffle_q}}=(\kY,{\tt conc},1_{Y^*},\Delta_{{\stuffle_q}},\epsilon)$
and $\calH_{{\stuffle_q}}^{\vee}=(\kY,{\stuffle_q},1_{Y^*},\Delta_{\tt conc},\epsilon)$
which are mutually dual bialgebras and, in fact, Hopf algebras because they are $\N$-graded by the weight, defined by
\begin{eqnarray}
\forall w=y_{i_1}\ldots y_{i_r}\in Y^+,&&(w)=i_1+\ldots+i_r.
\end{eqnarray}

\begin{lemma}[Friedrichs criterium]\label{lemme}
Let $S\in\kYY$ (for (2), we suppose in addition that $\scal{S}{1_{Y^*}}=1$). Then,
\begin{enumerate}
\item\label{1} $S$ is {\em primitive}, {\it i.e.} $\Delta_{\ministuffle_q}S=S\otimes1_{Y^*}+1_{Y^*}\otimes S$, if and only if,
for any $u,v\in Y^+$, $\scal S{u{\stuffle_q} v}=0$.
\item\label{2} $S$ is {\em group-like}, {\it i.e.} $\Delta_{\ministuffle_q}S=S\otimes S$, if and only if,
for any $u,v\in Y^+$, $\scal{S}{u{\stuffle_q} v}=\scal{S}{u}\scal{S}{v}$.
\end{enumerate}
\end{lemma}

\begin{proof}
The expected equivalence is due respectively to the following facts
\begin{eqnarray*}
&&\Delta_{\ministuffle_q}S=S\otimes1_{Y^*}+1_{Y^*}\otimes S-\scal S{1_{Y^*}\otimes1_{Y^*}}1_{Y^*}\otimes1_{Y^*}+\sum_{u,v\in Y^+}\scal S{u{\stuffle_q} v}u\otimes v,\\
&&\Delta_{\ministuffle_q}S=\sum_{u,v\in Y^*}\scal S{u{\stuffle_q} v}u\otimes v
\quad\mbox{and}\quad
S\otimes S=\sum_{u,v\in Y^*}\scal{S}{u}\scal{S}{v}u\otimes v.
\end{eqnarray*}
\end{proof}

\begin{lemma}\label{lemme2}
Let $S\in\kYY$ such that $\scal{S}{1_{Y^*}}=1$. Then, for the co-product $\Delta_{\ministuffle_q}$, $S$ is group-like if and only if $\log S$ is primitive.
\end{lemma}

\begin{proof}
Since $\Delta_{\ministuffle_q}$ and the maps  $T\mapsto T\otimes1_{Y^*},T\mapsto 1_{Y^*}\otimes T$ are continous homomorphisms then if $\log S$ is primitve then, by  Lemma \ref{lemme}, $\Delta_{\ministuffle_q}(\log S)=\log S\otimes1_{Y^*}+1_{Y^*}\otimes\log S$. Since $\log S\otimes1_{Y^*},1_{Y^*}\otimes\log S$ commute then
\begin{eqnarray*}
\Delta_{\ministuffle_q}S
&=&\Delta_{\ministuffle_q}(\exp(\log S))\\
&=&\exp(\Delta_{\ministuffle_q}(\log S))\\
&=&\exp(\log S\otimes1_{Y^*})\exp(1_{Y^*}\otimes\log S)\\
&=&(\exp(\log S)\otimes1_{Y^*})(1_{Y^*}\otimes\exp(\log S))\\
&=&S\otimes S.
\end{eqnarray*}
This means $S$ is group-like. The converse can be obtained in the same way.
\end{proof}

\begin{lemma}\label{L3}
Let $S_1,\ldots,S_n$ be {\em proper} formal power series in $\kYY$.
Let $P_1,\ldots,P_m$ be primitive elements in $\kY$, for the co-product $\Delta_{\ministuffle}$.
\begin{enumerate}
\item\label{1}  If $n>m$ then  $\scal{S_1{\stuffle_q}\ldots{\stuffle_q} S_n}{P_1\ldots P_m}=0.$
\item\label{2}  If $n=m$ then
\begin{eqnarray*}
\scal{S_1{\stuffle_q}\ldots{\stuffle_q} S_n}{P_1\ldots P_n}&=&
\sum_{\sigma\in\mathfrak{S}_n}\prod_{i=1}^n\scal{S_i}{P_{\sigma(i)}}.
\end{eqnarray*}
\item\label{3} If $n<m$ then, by considering the  language $\mathcal{M}$ over the new alphabet $\calA=\{a_1,\ldots,a_m\}$
\begin{eqnarray*}
\mathcal{M}&=&\{w\in\calA^*|w=a_{j_1}\ldots a_{j_{|w|}},j_1<\ldots<j_{|w|},|w|\ge1\}
\end{eqnarray*} 
and the morphism
$\mu:\Q\langle\calA\rangle\longrightarrow\kY$ given by, for any $i=1,\ldots,m,\mu(a_i)=P_i$, one has\label{support}~:
\begin{eqnarray*}
\scal{S_1{\stuffle_q}\ldots{\stuffle_q} S_n}{P_1\ldots P_m}
&=&\sum_{w_1,\ldots,w_m\in\mathcal{M}\atop\supp(w_1{\shuffle}\ldots{\shuffle} w_m)\ni a_1\ldots a_m}
\prod_{i=1}^n\scal{S_i}{\mu(w_i)}.
\end{eqnarray*}
\end{enumerate}
\end{lemma}

\begin{proof}
On the one hand, since the $P_i$'s are primitive then
\begin{eqnarray*}
\Delta_{\ministuffle_q}^{(n-1)}(P_i)&=&\sum_{p+q=n-1}1_{Y^*}^{\otimes p}\otimes P_i\otimes1_{Y^*}^{\otimes q}.
\end{eqnarray*}
On the other hand,
$$\Delta_{\ministuffle_q}^{(n-1)}(P_1\ldots P_m)
=\Delta_{\ministuffle_q}^{(n-1)}(P_1)\ldots\Delta_{\ministuffle_q}^{(n-1)}(P_m)$$
and
$$\scal{S_1{\stuffle_q}\ldots{\stuffle_q} S_n}{P_1\ldots P_m}
=\scal{S_1\otimes\ldots\otimes S_n}{\Delta_{\ministuffle_q}^{(n-1)}(P_1\ldots P_m)}.$$
Hence,
\begin{eqnarray*}
\scal{S_1{\stuffle_q}\ldots{\stuffle_q} S_n}{P_1\ldots P_m}
&=&\scal{\bigotimes_{i=1}^nS_i}{\prod_{i=1}^m\sum_{p+q=n-1}1_{Y^*}^{\otimes p}\otimes P_i\otimes1_{Y^*}^{\otimes q}}.
\end{eqnarray*}

\begin{enumerate}
\item  For $n>m$, by expanding $\Delta_{\ministuffle_q}^{(n-1)}(P_1)\ldots\Delta_{\ministuffle_q}^{(n-1)}(P_m)$,
one obtains a sum of tensors containing at least one factor equal to $1_{Y^*}$. 
For $i=1,..,n$, $S_i$ is proper and the result follows immediately.

\item  For $n=m$, since
\begin{eqnarray*}
\prod_{i=1}^n\Delta_{\ministuffle_q}^{(n-1)}(P_i)
&=&\sum_{\sigma\in\mathfrak{S}_n}\bigotimes_{i=1}^nP_{\sigma(i)}+Q,
\end{eqnarray*}
where $Q$ is sum of tensors containing at least one factor equal to $1$ and the $S_i$'s are proper  then
$\scal{S_1\otimes\ldots\otimes S_n}{Q}=0$.
Thus, the result follows.
\item  For $n<m$, since, for $i=1,..,n$, the power series $S_i$ is proper then the expected result follows by expanding the product
\begin{eqnarray*}
\prod_{i=1}^m\Delta_{\ministuffle_q}^{(n-1)}(P_i)&=&
\prod_{i=1}^m\sum_{p+q=n-1}1_{Y^*}^{\otimes p}\otimes P_i\otimes1_{Y^*}^{\otimes q}.
\end{eqnarray*}
\end{enumerate}
\end{proof}

\begin{definition}\label{pi1}
Let $\pi_1$ and  ${\check\pi_1}$ be the mutually adjoint projectors degree-preserving linear endomorphisms of $\kY$ given by, for any $w\in Y^+$,
\begin{eqnarray*}
\pi_1(w)&=&w+\sum_{k\ge2}\frac{(-1)^{k-1}}{k}
\sum_{u_1,\ldots,u_k\in Y^+}\langle w\bv u_1{\stuffle_q}\ldots{\stuffle_q} u_k\rangle u_1\ldots u_k,\\
{\check\pi_1}(w)&=&w+\sum_{k\ge2}\frac{(-1)^{k-1}}{k}
\sum_{u_1,\ldots,u_k\in Y^+}\langle w\bv u_1\ldots u_k\rangle u_1{\stuffle_q}\ldots{\stuffle_q} u_k.
\end{eqnarray*}
In particular, for any $y_k\in Y$, the polynomials $\pi_1(y_k)$ and ${\check\pi_1}(y_k)$ are given by
\begin{eqnarray*}
\pi_1(y_k)=y_k+\sum_{l\ge2}\frac{(-q)^{l-1}}{l}\sum_{j_1,\ldots,j_l\ge1\atop j_1+\ldots+j_l=k}y_{j_1}\ldots y_{j_l}
&\mbox{and}&
{\check\pi_1}(y_k)=y_k.
\end{eqnarray*}
\end{definition}

\begin{proposition}\label{logD}
Let $\calD_Y$ be the diagonal series over $Y$~:
\begin{eqnarray*}
\calD_Y&=&\sum_{w\in Y^*}w\otimes w.
\end{eqnarray*}
Then
\begin{enumerate}
\item  $\log\calD_Y=\Sum_{w\in Y^+}w\otimes\pi_1(w)=\Sum_{w\in Y^+}{\check\pi_1}(w)\otimes w.$

\item  For any $w\in Y^*$, we  have
\begin{eqnarray*}
w&=&\sum_{k\ge0}\frac1{k!}\sum_{u_1,\ldots,u_k\in Y^+}
\langle w\bv u_1{\stuffle_q}\ldots{\stuffle_q} u_k\rangle\pi_1(u_1)\ldots\pi_1(u_k)\\
&=&\sum_{k\ge0}\frac1{k!}\sum_{u_1,\ldots,u_k\in Y^+}
\langle w\bv u_1\ldots u_k\rangle{\check\pi_1}(u_1){\stuffle_q}\ldots{\stuffle_q}{\check\pi_1}(u_k).
\end{eqnarray*}
In particular, for any $y_s\in Y$, we have
\begin{eqnarray*}
y_s=\sum_{k\ge1}\frac{q^{k-1}}{k!}\sum_{s'_1+\cdots +s'_k=s}\pi_1(y_{s'_1})\ldots\pi_1(y_{s'_k})
&\mbox{and}&
y_s=\check\pi_1(y_s).
\end{eqnarray*}
\end{enumerate}
\end{proposition}

\begin{proof}
\begin{enumerate}
\item  Expanding by different ways the logarithm, it follows the results~:
\begin{eqnarray*}
\log\calD_Y
&=&\sum_{k\ge1}\frac{(-1)^{k-1}}{k}\biggl(\sum_{w\in Y^+}w\otimes w\biggr)^k\cr
&=&\sum_{k\ge1}\frac{(-1)^{k-1}}{k}\sum_{u_1,\ldots,u_k\in Y^+}
(u_1{\stuffle_q}\ldots{\stuffle_q} u_k)\otimes u_1\ldots u_k\cr
&=&\sum_{w\in Y^+}w\otimes\sum_{k\ge1}\frac{(-1)^{k-1}}{k}
\sum_{u_1,\ldots,u_k\in Y^+}\langle w\bv u_1{\stuffle_q}\ldots{\stuffle_q} u_k\rangle u_1\ldots u_k.\cr
\log\calD_Y
&=&\sum_{w\in Y^+}\sum_{k\ge1}\frac{(-1)^{k-1}}{k}
\sum_{u_1,\ldots,u_k\in Y^+}\langle w\bv u_1\ldots u_k\rangle u_1{\stuffle_q}\ldots{\stuffle_q} u_k\otimes w.
\end{eqnarray*}

\item  Since $\calD_Y=\exp(\log(\calD_Y))$ then, by the previous results, one  has separately,
\begin{eqnarray*}
\calD_Y
&=&\sum_{k\ge0}\frac1{k!}\biggl(\sum_{w\in Y^+}w\otimes\pi_1(w)\biggr)^k\\
&=&\sum_{k\ge0}\frac1{k!}\sum_{u_1,\ldots,u_k\in Y^+}
(u_1{\stuffle_q}\ldots{\stuffle_q} u_k)\otimes(\pi_1(u_1)\ldots\pi_1(u_k))\\
&=&\sum_{w\in Y^+}w\otimes\sum_{k\ge0}\frac1{k!}\sum_{u_1,\ldots,u_k\in Y^+}
\langle w\bv u_1{\stuffle_q}\ldots{\stuffle_q} u_k\rangle \pi_1(u_1)\ldots\pi_1(u_k).\\
\calD_Y
&=&\sum_{k\ge0}\frac1{k!}\sum_{u_1,\ldots,u_k\in Y^+}
({\check\pi_1}(u_1){\stuffle_q}\ldots{\stuffle_q}{\check\pi_1}(u_k))\otimes(u_1\ldots u_k)\\
&=&\sum_{w\in Y^+}\sum_{k\ge0}\frac1{k!}\sum_{u_1,\ldots,u_k\in Y^+}
\langle w\bv u_1\ldots u_k\rangle{\check\pi_1}(u_1){\stuffle_q}\ldots{\stuffle_q}{\check\pi_1}(u_k)\otimes w.
\end{eqnarray*}
\end{enumerate}
It follows then the expected result.
\end{proof}

\begin{lemma}\label{L2}
For any $w\in Y^+$, one has $\Delta_{\ministuffle_q}\pi_1(w)=\pi_1(w)\otimes1_{Y^*}+1_{Y^*}\otimes\pi_1(w)$.
\end{lemma}

\begin{proof}
Let $\alpha$ be the alphabet duplication isomorphism defined by, for any $\bar y\in\bar Y$, $\bar y=\alpha(y)$ 

Applying the tensor product of algebra isomorphisms $\alpha\otimes\mathrm{Id}$ to the diagonal series $\calD_Y$,
we obtain, by Lemma \ref{lemme}, a group-like element and then applying the logarithm of this element (or equivalently, applying $\alpha\otimes\pi_1$ to $\calD_Y$) we obtain $\calS$ which is, by Lemma \ref{lemme2}, a primitive element~:
\begin{eqnarray*}
(\alpha\otimes\mathrm{Id})\calD_Y=\sum_{w\in Y^*}\alpha(w)\;w
&\mbox{and}&
\calS=(\alpha\otimes\pi_1)\calD_Y=\sum_{w\in Y^*}\alpha(w)\;\pi_1(w).
\end{eqnarray*}
The two members of the identity $\Delta_{\ministuffle_q}\calS=\calS\otimes1_{Y^*}+1_{Y^*}\otimes\calS$ give respectively
\begin{eqnarray*}
\sum_{w\in Y^*}\alpha(w)\;\Delta_{\ministuffle_q}\pi_1(w)
&\mbox{and}&
\sum_{w\in Y^*}\alpha(w)\;\pi_1(w)\otimes1_{Y^*}+\sum_{w\in Y^*}\alpha(w)\;1_{Y^*}\otimes\pi_1(w).
\end{eqnarray*}
Since  $\{ w\}_{w\in\bar Y^*}$ is a basis for $\Q\pol{\bar Y}$ then identifying the coefficients in the previous expressions, we get $\Delta_{\ministuffle_q}\pi_1(w)=\pi_1(w)\otimes1_{Y^*}+1_{Y^*}\otimes\pi_1(w)$ meaning that $\pi_1(w)$ is primitive.
\end{proof}

\subsection{Pair of bases in duality on $q$-deformed stuffle algebra}
Let $\calP=\{P\in\QY\bv\Delta_{\ministuffle_q}P=P\otimes1_{Y^*}+1_{Y^*}\otimes P\}$ be the set of primitive polynomials \cite{B_Lie_II_III}.
Since, in virtue of Lemma \ref{L2}., $\mathrm{Im}(\pi_1)\subseteq \calP$, we can state the following

\begin{definition}\label{Pi}
Let $\{\Pi_l\}_{l\in \Lyn Y}$ be the family of $\calP$ and\footnote{Due to the fact this Hopf algebra is cocommutative and graded, then by the theorem of CQMM, $\kY\simeq \calU(\calP)$.} $\kY$ obtained as follows
$$\begin{array}{cccll}
\Pi_{y_k}&=&\pi_1(y_k)&\mbox{for}&k\ge1,\\
\Pi_{l}&=&[\Pi_s,\Pi_r]&\mbox{for}&l\in\Lyn X,\mbox{ standard factorization of }l=(s,r),\\
\Pi_{w}&=&\Pi_{l_1}^{i_1}\ldots \Pi_{l_k}^{i_k}
&\mbox{for}&w=l_1^{i_1}\ldots l_k^{i_k},l_1>\ldots>l_k,l_1\ldots,l_k\in\Lyn Y.
\end{array}$$
\end{definition}

\begin{proposition}\label{Ptriangulaire}
\begin{enumerate}
\item  For $l\in\Lyn Y$, the polynomial $\Pi_l$ is upper triangular and homogeneous in weight~:
\begin{eqnarray*}
\Pi_l&=&l+\sum_{v>l,(v)=(l)}c_vv,
\end{eqnarray*}
where for any $w\in Y^+$, $(w)$ denotes the weight of $w$ with $(y_k)=\deg(y_k)=k$.
\item The family $\{\Pi_w\}_{w\in Y^*}$ is upper triangular and homogeneous in weight~:
\begin{eqnarray*}
\Pi_w&=&w+\sum_{v>w,(v)=(w)}c_vv.
\end{eqnarray*}
\end{enumerate}
\end{proposition}

\begin{proof}
\begin{enumerate}
\item  Let us prove it by induction on the length of $l$~:
the result is immediate for $l\in Y$.
The result is suppose verified for any $l\in\Lyn Y\cap Y^k$ and $0\le k\le N$.
At $N+1$, by the standard factorization $(l_1,l_2)$ of $l$,
one has $\Pi_{l}=[\Pi_{l_1},\Pi_{l_2}]$ and $l_2l_1>l_1l_2=l$.
By induction hypothesis,
\begin{eqnarray*}
\Pi_{ l_1}=l_1+\sum_{v> l_1,(v)=(l_1)}c_vv
&\mbox{and}&
\Pi_{ l_2}=l_2+\sum_{u> l_2,(u)=(l_2)}d_uu,\\
&\Rightarrow&\Pi_{l}=l+\sum_{w> l,(w)=(l)}e_ww,
\end{eqnarray*}
getting $e_w$'s from $c_v$'s and $d_u$'s.


\item  Let $w=l_1\ldots l_k$, with $l_1\ge\ldots\ge l_k$ and $l_1,\ldots,l_k\in\Lyn Y$.
One has
\begin{eqnarray*}
\Pi_{l_i}=l_i+\sum_{v>l_i,(v)=(l_i)}c_{i,v}v
&\mbox{and}&
\Pi_w=l_1\ldots l_k+\sum_{u>w,(u)=(w)}d_uu,
\end{eqnarray*}
where the $d_u$'s are obtained from the $c_{i,v}$'s. Hence, the family $\{\Pi_w\}_{w\in Y^*}$  is upper triangular and  homogeneous in weight. As the grading by weight is in finite dimensions, this family is a basis of $\kY$. 
\end{enumerate}
\end{proof}

\begin{definition}\label{Sigma}
Let $\{\Sigma_w\}_{w\in Y^*}$ be the family of the quasi-shuffle algebra (viewed as a $\Q$-module) obtained by duality with $\{\Pi_w\}_{w\in Y^*}$~:
\begin{eqnarray*}
\forall u,v\in Y^*,\quad\scal{\Sigma_{v}}{\Pi_u}&=&\delta_{u,v}.
\end{eqnarray*}
\end{definition}
\begin{proposition}\label{Striangulaire}
The family $\{\Sigma_w\}_{w\in Y^*}$ is lower triangular and homogeneous in weight. In other words,
\begin{eqnarray*}
\Sigma_w&=& w+\sum_{v<w,(v)=(w)}d_vv.
\end{eqnarray*}
\end{proposition}

\begin{proof}
By duality with $\{\Pi_w\}_{w\in Y^*}$ (see Proposition \ref{Pi}), we get the expected result.
\end{proof}

\begin{theorem}
\begin{enumerate}
\item The family $\{\Pi_l\}_{l\in\Lyn Y}$ forms a basis of $\calP$.
\item The family $\{\Pi_w\}_{w\in Y^*}$ forms a basis of $\kY$.
\item The family $\{\Sigma_w\}_{w\in Y^*}$ generate freely the quasi-shuffle algebra.
\item  The family $\{\Sigma_l\}_{l\in\Lyn Y}$ forms a transcendence basis of $(\kY,{\stuffle_q})$.
\end{enumerate}
\end{theorem}

\begin{proof}
The family $\{\Pi_l\}_{l\in\Lyn Y}$ of primitive upper triangular  homogeneous in weight polynomials is free and
the first result follows. The second is a direct consequence of the Poincar\'e-Birkhoff-Witt theorem.
By the Cartier-Quillen-Milnor-Moore theorem, we get the third one and the last one is obtained
as consequence of the constructions of $\{\Sigma_l\}_{l\in\Lyn Y}$ and $\{\Sigma_w\}_{w\in Y^*}$.
\end{proof}

To decompose any letter $y_s\in Y$ in the basis $\{\Pi_w\}_{w\in Y^*}$, one can use its expression in Proposition \ref{logD}.

Now, using the mutually adjoint projectors  $\pi_1$ and ${\check\pi_1}$ given in Definition \ref{pi1} and are determinded by Proposition \ref{logD}, let us clarify the basis $\{\Sigma_w\}_{w\in Y^*}$ and then the transcendence basis $\{\Sigma_l\}_{l\in\Lyn Y}$ of the quasi-shuffle algebra $(\kY,{\stuffle_q},1_{Y^*})$ as follows

\begin{proposition}\label{checkSigma}
We have
\begin{enumerate}
\item  For $w=1_{Y^*}$, $\Sigma_{w}=1$.
\item  For any $w=l_1^{i_1}\ldots l_k^{i_k}$, with $l_1,\ldots,l_k\in\Lyn Y$ and $l_1>\ldots>l_k$,
\begin{eqnarray*}
\Sigma_w&=&\Frac{\Sigma_{l_1}^{{\stuffle_q} i_1}{\stuffle_q}\ldots{\stuffle_q}\Sigma_{l_k}^{{\stuffle_q} i_k}}{i_1!\ldots i_k!}.
\end{eqnarray*}
\item  For any $y\in Y$,
\begin{eqnarray*}
\Sigma_{y}=y={\check\pi_1}(y).
\end{eqnarray*}
\end{enumerate}
\end{proposition}

\begin{proof}
\begin{enumerate}
\item  Since $\Pi_{1_{Y^*}}=1$ then $\Sigma_{1_{Y^*}}=1$.

\item  Let $u=u_1\ldots u_n=l_1^{i_1}\ldots l_k^{i_k},v=v_1\ldots v_m=h_1^{j_1}\ldots h_p^{j_p}$ with
$l_1\ldots,l_k,\allowbreak h_1,\ldots,h_p,\allowbreak u_1,\ldots,u_n$ and $v_1,\ldots,v_m\in\Lyn Y,\allowbreak l_1>\ldots>l_k,\allowbreak h_1>\ldots>h_p,\allowbreak u_1\ge\ldots\ge u_n$ and $v_1\ge\ldots\ge v_m$ 
and  $i_1+\ldots+i_k=n,\allowbreak j_1+\ldots+j_p=m$. Hence, if $m\ge2$ (resp. $n\ge2$)
then $v\notin\Lyn Y$ (resp. $u\notin\Lyn Y$).

Since
\begin{eqnarray*}
\scal{\Sigma_{u_1}{\stuffle_q}\ldots{\stuffle_q}\Sigma_{u_n}}{\prod_{i=1}^n\Pi_{v_i}}
&=&\scal{\Sigma_{u_1}\otimes\ldots\otimes\Sigma_{u_n}}{\Delta_{\ministuffle_q}^{(n-1)}(\Pi_{v_1}\ldots\Pi_{v_m})}
\end{eqnarray*}
then many cases occur~:
\begin{enumerate}
\item Case $n>m$. By Lemma \ref{L3}(\ref{1}), one has
\begin{eqnarray*}
\scal{\Sigma_{u_1}{\stuffle_q}\ldots{\stuffle_q}\Sigma_{u_n}}{\Pi_{v_1}\ldots\Pi_{v_m}}=0.
\end{eqnarray*}

\item Case $n=m$. By Lemma \ref{L3}(\ref{2}), one has
\begin{eqnarray*}
\scal{\Sigma_{u_1}{\stuffle_q}\ldots{\stuffle_q}\Sigma_{u_n}}{\prod_{i=1}^n\Pi_{v_i}}
&=&\sum_{\sigma\in\Sigma_n}\prod_{i=1}^n\scal{\Sigma_{u_i}}{\Pi_{v_{\sigma(i)}}}\\
&=&\sum_{\sigma\in\Sigma_n}\prod_{i=1}^n\delta_{u_i,v_{\sigma(i)}}.
\end{eqnarray*}
Thus, if $u\neq v$ then $(u_1,\ldots,u_n)\neq(v_1,\ldots,v_n)$ then the second member is vanishing else,
{\it i.e.} $u=v$, the second member equals $1$ because the factorization by Lyndon words is unique.

\item Case $n<m$. By Lemma \ref{L3}(\ref{3}), let us consider the following language over the new alphabet
$\calA:=\{a_1,\ldots,a_m\}$~:
\begin{eqnarray*}
\mathcal{M}&=&\{w\in\calA^*|w=a_{j_1}\ldots a_{j_{|w|}},j_1<\ldots<j_{|w|},|w|\ge1\},
\end{eqnarray*}
and the morphism $\mu:\Q\langle\calA\rangle\longrightarrow\kY$ given by, for any $i=1,\ldots,m,\mu(a_i)=\Pi_{v_i}$. We get~:
\begin{eqnarray*}
\scal{\Sigma_{u_1}{\stuffle_q}\ldots{\stuffle_q}\Sigma_{u_n}}{\prod_{i=1}^n\Pi_{v_i}}
&=&\sum_{w_1,\ldots,w_n\in\mathcal{M}\atop\supp(w_1{\shuffle}\ldots{\shuffle} w_n)\ni a_1\ldots a_m}
\prod_{i=1}^n\scal{\Sigma_{u_i}}{\mu(w_i)}\\
&=&0.
\end{eqnarray*}
Because in the right side of the first equality, on the one hand, there is at least one $w_i,|w_i|\ge2$, corresponding to
$\mu(w_i)=\Pi_{v_{j_1}}\ldots\Pi_{v_{j_{|w_i|}}}$ such that
$v_{j_1}\ge\ldots\ge v_{j_{|w_i|}}$ and on the other hand, 
$\nu_i:=v_{j_1}\ldots v_{j_{|w_i|}}\notin\Lyn Y$ and $u_i\in\Lyn Y$.
\end{enumerate}
By consequent,
\begin{eqnarray*}
\scal{\Sigma_{u}}{\Pi_v}
&=&\frac{1}{i_1!\ldots i_k!}
\scal{\Sigma_{l_1}^{{\stuffle_q} i_1}{\stuffle_q}\ldots{\stuffle_q}\Sigma_{l_k}^{{\stuffle_q} i_k}}{\Pi_{h_1}^{j_1}\ldots\Pi_{h_p}^{j_p}}\\
&=&\delta_{u,v}.
\end{eqnarray*}
\item For any $y\in Y$, by Proposition \ref{Striangulaire}, $\Sigma_{y}=y={\check\pi_1}(y)$. The directe computation prove that, for any $w\in Y^*$ and for any $y\in Y$, one has $\scal{\Pi_w}{\Sigma_y}=\delta_{w,y}$.
\end{enumerate}
\end{proof}

\begin{proposition}\label{D_Y}
\begin{enumerate}
\item For $w\in Y^+$, the polynomial $\Sigma_w$ is proper and homogeneous of degree $(w)$, for $\deg(y_i)=i$, and with rational positive coefficients.
\item  $\calD_Y=\Sum_{w\in Y^*}\Sigma_w\otimes\Pi_w
=\Prod_{l\in\Lyn Y}^{\searrow}\exp(\Sigma_l\otimes\Pi_l)$.

\item  The family $\Lyn Y$ forms a transcendence basis 
of the quasi-shuffle algebra and the family of proper polynomials of rational positive coefficients defined by,
for any $w=l_1^{i_1}\ldots l_k^{i_k}$ with $l_1>\ldots>l_k$ and $l_1,\ldots,l_k\in\Lyn Y$,
\begin{eqnarray*}
\chi_w&=&\frac{1}{i_1!\ldots i_k!}l_1^{{\stuffle_q} i_1}{\stuffle_q}\ldots{\stuffle_q} l_k^{{\stuffle_q} i_k}
\end{eqnarray*}
forms a basis of the quasi-shuffle algebra.
\item  Let $\{\xi_w\}_{w\in Y^*}$ be the basis of the envelopping algebra $\calU(\LQX)$ obtained by duality with $\{\chi_w\}_{w\in Y^*}$~:
\begin{eqnarray*}
\forall u,v\in Y^*,\quad\scal{\chi_{v}}{\xi_u}&=&\delta_{u,v}.
\end{eqnarray*}
Then the family $\{\xi_l\}_{l\in\Lyn Y}$ forms  a  basis of the free Lie algebra $\LQY$.
\end{enumerate}
\end{proposition}

\begin{proof}
\begin{enumerate}
\item The proof can be done by induction on the length of $w$ using the fact that the product ${\stuffle_q}$ conserve the property, l'homogenity and rational positivity of the coefficients.

\item  Expressing $w$ in  the basis $\{\Sigma_w\}_{w\in Y^*}$ of the quasi-shuffle algebra and then in the basis $\{\Pi_w\}_{w\in Y^*}$ of the envelopping algebra, we obtain successively
\begin{eqnarray*}
\calD_Y
&=&\sum_{w\in Y^*}\biggl(\sum_{u\in Y^*}\scal{\Pi_u}{w}\Sigma_u\biggr)\otimes w\cr
&=&\sum_{u\in Y^*}\Sigma_u\otimes\biggl(\sum_{w\in Y^*}\scal{\Pi_u}{w}w\biggr)\cr
&=&\sum_{u\in Y^*}\Sigma_u\otimes\Pi_u\cr
&=&\sum_{l_1>\ldots>l_k\atop i_1,\ldots,i_k\ge1}
\Frac{1}{i_1!\ldots i_k!}\Sigma_{l_1}^{{\stuffle_q} i_1}{\stuffle_q}\ldots{\stuffle_q}\Sigma_{l_k}^{{\stuffle_q} i_k}
\otimes\Pi_{l_1}^{i_1}\ldots\Pi_{l_k}^{i_k}\cr
&=&\prod_{l\in\Lyn Y}^{\searrow}\sum_{i\ge0}\frac{1}{i!}\Sigma_l^{{\stuffle_q} i}\otimes\Pi_l^{i}\cr
&=&\prod_{l\in\Lyn Y}^{\searrow}\exp(\Sigma_l\otimes\Pi_l).
\end{eqnarray*}

\item  For  $w=l_1^{i_1}\ldots l_k^{i_k}$ with $l_1,\ldots,l_k\in\Lyn Y$ and $l_1>\ldots>l_k$,
by Proposition \ref{Ptriangulaire}, the proper polynomial  of positive coefficients $\Sigma_w$ is lower triangular~:
\begin{eqnarray*}
\Sigma_w
&=&\frac{1}{i_1!\ldots i_k!}\Sigma_{l_1}^{{\stuffle_q} i_1}{\stuffle_q}\ldots{\stuffle_q}\Sigma_{l_k}^{{\stuffle_q} i_k}\\
&=&w+\sum_{v<w,(v)=(w)}c_vv.
\end{eqnarray*}
In particular, for any $l_j\in\Lyn Y$, $\Sigma_{l_j}$ is lower triangular~:
\begin{eqnarray*}
\Sigma_{l_j}&=&l_j+\sum_{v<l_j,(v)=(l_j)}c_vv.
\end{eqnarray*}
Hence, $\Sigma_w=\chi_w+\chi'_w$,
where $\chi'_w$ is a proper polynomial of $\kY$ of rational positive coefficients.
We deduce then the support of $\chi_w$ contains words which are less than $w$ and $\scal{\chi_w}{w}=1$.
Thus, the proper polynomial $\chi_w$ of rational positive coefficients is lower triangular~:
\begin{eqnarray*}
\chi_w&=&w+\sum_{v<w,(v)=(w)}c_vv,\\
\Rightarrow\quad\forall l\in\Lyn Y,\quad\chi_l&=&l+\sum_{v<l,(v)=(l)}c_vv.
\end{eqnarray*}
It follows then expected results.

\item  By duality, for $w\in Y^*$, the proper polynomial  $\xi_w$ is upper triangular.
In particular, for any $l\in\Lyn Y$, the proper polynomial $\xi_l$ is upper triangular~:
\begin{eqnarray*}
\xi_l&=&l+\sum_{v>l,(v)=(l)}d_vv.
\end{eqnarray*}
Hence, the family  $\{\xi_l\}_{l\in\Lyn Y}$ is free and its elements verify an analogous of the generalized criterion of Friedrichs~:
\begin{itemize}
\item for $w\in\Lyn Y$, one has $\scal{\chi_w}{\xi_l}=\delta_{w,l}$,
\item for $w=l_1\ldots l_n\notin\Lyn Y$ with  $l_1,\ldots,l_n\in\Lyn Y$ and $l_1\ge\ldots\ge l_n$, one has  (since $l\in\Lyn Y$)
\begin{eqnarray*}
\scal{\chi_{l_1}{\stuffle_q}\ldots{\stuffle_q}\chi_{l_n}}{\xi_l}=\scal{\chi_w}{\xi_l}=0.
\end{eqnarray*}
\end{itemize}
The polynomials $\xi_l$'s are primitive. Actually, we have
\begin{eqnarray*}
\Delta_{\ministuffle_q}\xi_l
&=&\sum_{u\in Y^+}\scal{u{\stuffle_q}1_{Y^*}}{\xi_l}u\otimes1
+\sum_{v\in Y^+}\scal{1_{Y^*}{\stuffle_q} v}{\xi_l}1\otimes v
+\sum_{u,v\in Y^+}\scal{u{\stuffle_q} v}{\xi_l}u\otimes v\\
&+&\scal{1_{Y^*}{\stuffle_q}1_{Y^*}}{\xi_l}1\otimes1\cr
&=&\xi_l\otimes1+1\otimes\xi_l.
\end{eqnarray*}
Because, after decomposing the words $u$ and $v$ on the transcendence basis $\{\chi_l\}_{l\in\Lyn Y}$ and by the previous fact, the third sum is vanishing. The last one is also vanishing since the $\xi_l$'s are proper. Hence, it follows the expected result.\end{enumerate}
\end{proof}

\subsection{Determination of $\{\Sigma_l\}_{l\in\Lyn Y}$}
Following \cite{reutenauer}, we call a \emph{standard sequence} of Lyndon words to be a sequence
\begin{eqnarray}\label{StandardSequence}
 S=(l_1,\cdots,l_k), k\ge1
\end{eqnarray}
if for all $i$, either $ l_i$ to be a letter or the standard factorization $\sigma(l_i)=(l_i',l_i'')$ and $l_i''\ge l_{i+1},\cdots,l_n$. Note that a decreasing sequence of Lyndon words is also a standard sequence. A \emph{rise}  of a sequence $S$ is an index $i$ such that $l_i<l_{i+1}$. A \emph{legal rise} of sequence $S$ is a rise of $i$ such that $l_{i+1}\ge l_{i+2},\cdots,l_k$;
 with the legal rise $i$, we define
\begin{eqnarray}\label{labelleft}
     \lambda_i(S)=(l_1,\cdots,l_{i-1},l_il_{i+1},l_{i+2},\cdots,l_n)
     &\mbox{and}&
     \rho_i(S)=(l_1,\cdots,l_{i-1},l_{i+1},l_{i},l_{i+2},\cdots,l_n)
\end{eqnarray}
We denote $S\Rightarrow {T}$ if $T=\lambda_i(S)$ or ${T}=\rho_i(S)$ for some legal rise $i$; and $S\stackrel{*}{\Rightarrow} T$, transitive closure of $\Rightarrow$.

A \emph{derivation tree} $\mathcal{T}(S)$ of $S$ to be a labelled rooted tree with the following properties~: if $S$ is decreasing, then $\mathcal{T}(S)$ is reduced to its root, labelled $S$; if not, $\mathcal{T}(S)$ is the tree with root labelled $S$, with left and right immediate subtree $\mathcal{T}(S')$ and $\mathcal{T}(S'')$, where $S'=\lambda_i(S)$, $S''=\rho_i(S)$ for some legal rise $i$ of $S$; we define $\Pi(S)=\Pi_{l_1}\ldots \Pi_{l_n}$ ($\Pi(S)\neq \Pi_{l_1\ldots l_k}$ because $l_1,\cdots,l_k$ can be not a decreasing sequence).
   
Conversely, we call a \emph{fall} of sequence $S$ is an index $i$ such that $l_1,\cdots,l_i \in Y, l_i>l_{i+1}$. We define
\begin{eqnarray}
\rho^{-1}_i(S)=(l_1,\cdots,l_{i+1},l_i,\cdots,l_n).
\end{eqnarray}

We call a \emph{landmark} of sequence $S$ is an index $i$ such that $l_1,\cdots,l_{i-1}\in Y, l_i\in Y^*\setminus Y$, and we define 
\begin{eqnarray}
\lambda^{-1}_i(S)=(l_1,\cdots,l_{i-1},l_i',l_i'',l_{i+1},\cdots,l_n),
\end{eqnarray}
where $\sigma(l_i)=(l_i',l_i'')$. We will denote by $S\Leftarrow T$ if  $T=\rho^{-1}_i(S)$ or $T=\lambda^{-1}_i(S)$ for some fall or landmark $i$; and $S\stackrel{*}{\Leftarrow} T$, transitive closure of $\Leftarrow$.

Similarly, we call the conversely derivation tree $\mathcal{T}^{-1}(S)$ with root labelled $S$, with left and right immediate subtree $\mathcal{T}^{-1}(S')$ and $\mathcal{T}^{-1}(S'')$, where $S'=\rho^{-1}_i(S)$ for some fall $i$, $S''=\lambda^{-1}_i(S)$ for some landmark $i$.

\begin{lemma}\label{derivationtree}
For each standard sequence $S$, $\Pi(S)$ is the sum of all $\Pi(T)$ for $T$ a leaf in a fixed derivation tree of $S$. 
\end{lemma}

\begin{proof}
This is a consequence of the definitions of $\lambda_i(S)$ and $\rho_i(S)$ on  (\ref{labelleft}), of $\mathcal{T}(S)$ and $\Pi(S)$, and of the identity
$\Pi_{l_i}\Pi_{l_{i+1}}=[\Pi_{l_i},\Pi_{l_{i+1}}]+\Pi_{l_{i+1}}\Pi_{l_i}=\Pi_{l_il_{i+1}}+\Pi_{l_{i+1}}\Pi_{l_i}$.
\end{proof}

\begin{example}
$\Pi(y_4,y_2,y_1)=\Pi_{y_4y_2y_1}+\Pi_{y_2y_1}\Pi_{y_4}+\Pi_{y_4y_1y_2}+\Pi_{y_2}\Pi_{y_4y_1}+\Pi_{y_1}\Pi_{y_4y_2}+\Pi_{y_1}\Pi_{y_2}\Pi_{y_4},$ we can see the following diagram (note that $y_4<y_2<y_1$)
\begin{figure}[ht]
\begin{center}
\includegraphics[height=4.8cm]{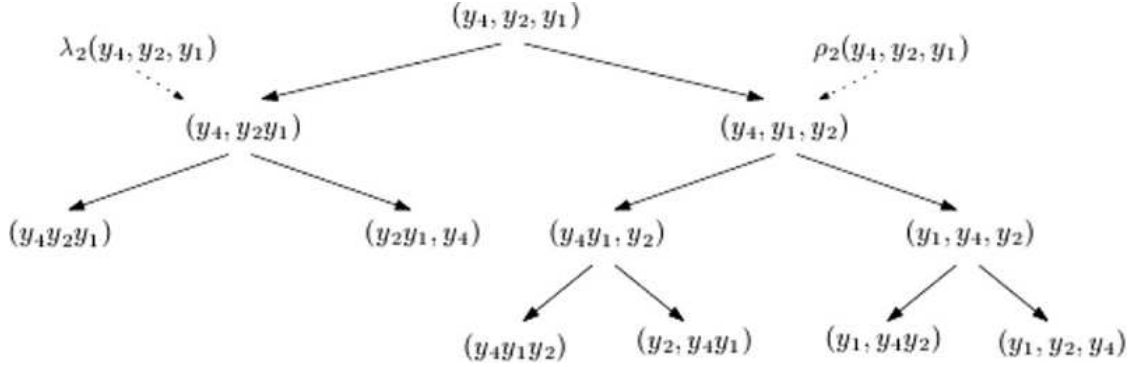}
\end{center}
\caption{ Derivation tree $\mathcal{T}(y_4,y_2,y_1)$}
\end{figure}
\end{example}

\begin{proposition}\label{increasinglyndon}
\begin{enumerate}
\item For any Lyndon word $y_{s_1}\ldots y_{s_k}$, we have
\begin{eqnarray*}\label{SigmaOfLyndonWord}
\Sigma_{y_{s_1}\ldots y_{s_k}}
&=&\sum_{{\{s'_1,\cdots,s'_i\}\subset \{s_1,\cdots,s_k\}, l_1\ge\cdots\ge l_n\in\Lyn Y}\atop (y_{s_1}\cdots y_{s_k})
\stackrel{*}{\Leftarrow} (y_{s'_1},\cdots,y_{s'_n},l_1,\cdots,l_n)}{\frac{q^{i-1}}{i!}y_{s'_1+\cdots+s'_i}
\Sigma_{l_1\cdots l_n}}.
\end{eqnarray*}
\item In special case, if $y_{s_1}\le\cdots\le y_{s_k}$ then
\begin{eqnarray*}
\Sigma_{y_{s_1}\ldots y_{s_k}}
&=&\sum_{i=1}^k{\frac{q^{i-1}}{i!}y_{s_1+\cdots+s_i}\Sigma_{y_{s_{i+1}}\ldots y_{s_k}}}.
\end{eqnarray*}
\end{enumerate}
\end{proposition}

\begin{proof}
At first, we remark this Proposition is equivalent to saying that for any word $u$ and any letter $y_s$,
\begin{eqnarray*}
\scal{\Sigma_{y_{s_1}\ldots y_{s_k}}}{y_su}
&=&\sum_{{\{s'_1,\cdots,s'_i\}\subset \{s_1,\cdots,s_k\}, l_1\ge\cdots\ge l_n\in\Lyn Y}\atop (y_{s_1}\cdots y_{s_k})
\stackrel{*}{\Leftarrow} (y_{s'_1},\cdots,y_{s'_n},l_1,\cdots,l_n)}
\frac{q^{i-1}}{i!}\delta_{s'_1+\cdots+s'_i,s}\scal{\Sigma_{l_1\ldots l_n}}{u}.
 \end{eqnarray*}
One has
\begin{eqnarray*}
u=\sum_{w\in Y^*}\scal{\Sigma_w}{u}\Pi_w,
\end{eqnarray*}
multiplying the two members by $y_s$ and by Proposition \ref{logD}, one obtains
\begin{eqnarray*}
y_su&=&\sum_{w\in Y^*}\scal{\Sigma_w}{u}\biggl(\sum_{i\ge1}\frac{q^{i-1}}{i!}
\sum_{s'_1+\cdots +s'_i=s}\Pi_{y_{s'_1}}\ldots\Pi_{y_{s'_i}}\biggr)\Pi_w\cr
&=&\sum_{w\in Y^*}
\scal{\Sigma_w}{u}\sum_{i\ge1}\frac{q^{i-1}}{i!}\sum_{s'_1+\cdots +s'_i=s}\Pi_{y_{s'_1}}\ldots\Pi_{y_{s'_i}}\Pi_w,\\
\Rightarrow
\scal{\Sigma_{y_1\ldots y_k}}{y_su}
&=& \sum_{w\in Y^*}\scal{\Sigma_w}{u}\sum_{i\ge1}\frac{q^{i-1}}{i!}
\sum_{s'_1+\cdots +s'_i=s}\scal{\Sigma_{y_1\cdots y_k}}{\Pi_{y_{s'_1}}\ldots\Pi_{y_{s'_i}} \Pi_w}.
\end{eqnarray*}
For each $w$ fixed, we write $w$ form factorization of Lyndon words $w=l_1\ldots l_n, l_1\ge\cdots\ge l_n$, then we have $S:=(y_{s'_1},\cdots,y_{s'_i},l_1,\cdots,l_n)$ is a standard sequence, so we obtain from Lemma \ref{derivationtree} \begin{eqnarray*}
\Pi(S)=\Pi(y_{s'_1},\cdots,y_{s'_i},l_1,\cdots,l_n)=\sum_{S\stackrel{*}{\Rightarrow}T}{\alpha_T\Pi(T)}.
\end{eqnarray*}
Consequently,
\begin{eqnarray*}
\scal{\Sigma_{y_1\ldots y_k}}{y_su}
&=&\sum_{l_1\ge\cdots\ge l_n\in\Lyn Y}\scal{\Sigma_{l_1\ldots l_n}}{u} \sum_{i\ge1}\frac{q^{i-1}}{i!}
\sum_{s'_1+\cdots +s'_i=s\atop (y_{s_1},\cdots,y_{s_i},l_n,\cdots,l_n)\stackrel{*}{\Rightarrow}T}\alpha_T\scal{\Sigma_{y_1\ldots y_k}}{\Pi(T)}.
\end{eqnarray*}
Note that, the leaves $T$'s of derivation tree $\mathcal{T}(S)$ are decreasing sequences of Lyndon words with length $\ge2$ except leaves form $T=(l)$, where $l\in\Lyn Y$. Therefore $\langle \Sigma_{y_1\ldots y_k}| \Pi(T)\rangle\neq 0$ if $T=(y_{s_1}\ldots y_{s_k})$. By maps $\rho^{-1}$ and $\lambda^{-1}$, we construct a conversely derivation tree from the standard sequence of one Lyndon word $S=(y_{s_1}\ldots y_{s_k})$, we take standard sequences form $(y_{s1},\cdots,y_{s_i},l_n,\cdots,l_n), i\ge1$; at that time, for each $S$ of these sequences, we get unique leaf $T=(y_{s_1}\ldots y_{s_k})$ in the derivation tree $\mathcal{T}(S)$, it mean $\alpha_T=1$. Thus, we get the expected result.
  
In other words, if $y_{s_1}\le\cdots\le y_{s_k}$ then the standard sequence $(y_{s_1}\ldots y_{s_k})$ may only be a leaf of a derivation tree $\mathcal{T}(S)$ after applying map $\lambda_i$ more times, we imply that
$\scal{ \Sigma_{y_{s_1}\ldots y_{s_k}}}{\Pi_{y_{s'_1}}\ldots\Pi_{y_{s'_i}} \Pi_w}\neq0$ if and only if
$y_{s_1}\ldots y_{s_k}=y_{s_1'}\ldots y_{s'_i}l_1\ldots l_n$, then $y_{s_1}=y_{s'_1},\cdots, y_{s'_i}=y_{s_i}$ and $y_{s_{i+1}}\ldots y_{s_k}=l_1\ldots l_n$. Hence
\begin{eqnarray*}
\langle \Sigma_{y_{s_1}\ldots y_{s_k}}|\Pi_{y_{s'_1}}\ldots\Pi_{y_{s'_i}} \Pi_w\rangle
&=&\delta_{s_1+\cdots+s_i,s}\delta_{y_{s_{i+1}}\ldots y_{s_k},w},
\end{eqnarray*}
we thus get
\begin{eqnarray*}
\scal{\Sigma_{y_{s_1}\ldots y_{s_k}}}{y_su}
&=&\frac{q^{i-1}}{i!}\delta_{s_1+\cdots+s_i,s}\scal{ \Sigma_{y_{s_{i+1}}\ldots y_{s_k}}}{u}.
\end{eqnarray*}
\end{proof}

\subsection{Examples with Maple}
\begin{eqnarray}
\Pi_{y_1}&=&y_1,\\
\Pi_{y_2}&=&y_2-\frac{q}{2}y_1^2,\\
\Pi_{y_2y_1}&=&y_2y_1-y_1y_2,\\
\Pi_{y_3y_1y_2}
&=&y_3y_1y_2-\frac{q}{2}y_3y_1^3-qy_2y_1^2y_2+\frac{q^2}4y_2y_1^4-y_1y_3y_2+\frac{q}2y_1y_3y_1^2\cr
&+&\frac{q}{2}y_1^2y_2^2-\frac{q^2}2y_1^2y_2y_1^2-y_2y_3y_1+\frac{q}2y_2^2y_1^2+y_2y_1y_3+\frac{q}2y_1^2y_3y_1-\frac{q}2y_1^3y_3+\frac{q^2}4y_1^4y_2,\\
\Pi_{y_3y_1y_2y_1}
 &=&y_3y_1y_2y_1-y_3y_1^2y_2-\frac{q}2y_2y_1^2y_2y_1-y_1y_3y_2y_1+y_1y_3y_1y_2+\frac{q}2y_1^2y_2^2y_1\cr
&-&\frac{q}2y_1^2y_2y_1y_2-y_2y_1y_3y_1+\frac{q}2y_2y_1y_2y_1^2+y_2y_1^2y_3+y_1y_2y_3y_1\cr
&-&\frac{q}2y_1y_2^2y_1^2-y_1y_2y_1y_3+\frac{q}2y_1y_2y_1^2y_2.\\
\cr
\Sigma_{y_1}&=&y_1,\\
\Sigma_{y_2}&=&y_2,\\
\Sigma_{y_2y_1}&=&y_2y_1+\frac{q}2y_3,\\
\Sigma_{y_3y_2y_1}&=&y_3y_1y_2+y_3y_2y_1+qy_3^2+\frac{q}2y_4y_2+\frac{q^2}{3}y_6+\frac{q}2y_5y_1,\\
\Sigma_{y_3y_1y_2y_1}
&=&2y_3y_2y_1^2+qy_3y_2^2+y_3y_1y_2y_1+\frac{3q}2y_3^2y_1+\frac{q}2y_3y_1y_3+\frac{q^2}2y_3y_4+\frac{q}2y_4y_2y_1\cr
&+&\frac{q^2}4y_4y_3+qy_5y_1^2+\frac{q^2}2y_5y_2+\frac{q^2}2y_6y_1+\frac{q^3}8y_7.
\end{eqnarray}
\section{Conclusion}

Since the pioneering works of Sch\"utzenberger and Reutenauer \cite{schutz,reutenauer}, the question of computing bases in duality (maybe at the cost of a more cumbersome procedure, but without inverting a Gram matrix) remained open in the case of cocommutative deformations of the shuffle product. We have given such a procedure, based on the computation of $\log_*(I)$ on the letters which allows a great simplification for an interpolation between shuffle and stuffle products (this interpolation reduces to the shuffle for $q=0$ and the stuffle for $q=1$). Our algorithm boils down to the classical one in the case when $q=0$.  In the next framework, this product will be continuously deformed, in the most general way but still commutative (see \cite{Enjalbert-HNM} for examples).

\end{document}